\documentclass{amsart}

\usepackage{amssymb}  
\usepackage{enumerate}
\usepackage[all]{xy}

\usepackage{color}
\definecolor{darkgreen}{rgb}{0, 0.5, 0}

\usepackage[pdftex,colorlinks=true,citecolor=darkgreen,backref,%
            pdftitle={Curves on the surface of cuboids},%
            pdfauthor={Michael Stoll and Damiano Testa}]{hyperref}

\newcommand{\Z}{{\mathbb Z}}
\newcommand{\Q}{{\mathbb Q}}
\newcommand{\F}{{\mathbb F}}
\newcommand{\C}{{\mathbb C}}
\newcommand{\BP}{{\mathbb P}}
\newcommand{\To}{\longrightarrow}
\newcommand{\Pic}{\operatorname{Pic}}
\newcommand{\rank}{\operatorname{rank}}
\newcommand{\sing}{\text{\rm sing}}
\newcommand{\eps}{\varepsilon}

\newcommand{\cE}{\mathcal{E}}
\newcommand{\cG}{\mathcal{G}}
\newcommand{\cO}{\mathcal{O}}
\newcommand{\cQ}{\mathcal{Q}}

\newcommand{\Qbar}{\overline{\Q}}
\newcommand{\Gal}{\operatorname{Gal}}
\newcommand{\Aut}{\operatorname{Aut}}
\newcommand{\Sym}{\operatorname{Sym}}
\newcommand{\PSL}{\operatorname{PSL}}
\renewcommand{\H}{\operatorname{H}}

\newtheorem{theorem}{Theorem}[section]
\newtheorem{prop}[theorem]{Proposition}
\newtheorem{lemma}[theorem]{Lemma}
\newtheorem{cor}[theorem]{Corollary}
\newtheorem{question}[theorem]{Question}

\theoremstyle{definition}
\newtheorem{definition}[theorem]{Definition}

\theoremstyle{remark}

\numberwithin{equation}{section}

\usepackage[alphabetic,backrefs,lite]{amsrefs} 

\begin{document}

\title{Curves on the surface of cuboids}

\author{Michael Stoll}
\address{Department of Mathematics,
         University of Bayreuth,
         95440 Bayreuth, Germany}
\email{Michael.Stoll@uni-bayreuth.de}

\author{Damiano Testa}
\address{Mathematics Institute,
         Zeeman Building,
         University of Warwick,
         Coventry CV4 7AL, England}
\email{D.Testa@warwick.ac.uk}


\date{}

\begin{abstract}
  We study the surface~$\bar{S}$ parametrizing cuboids: it is defined by the
  equations relating the sides, face diagonals and long diagonal
  of a rectangular box. It is an open problem whether a `rational cuboid'
  exists, i.e., a rectangular box all of whose sides, face diagonals
  and long diagonal have (positive) rational length. The question is
  equivalent to the existence of nontrivial rational points on~$\bar{S}$.

  The surface~$\bar S$ has 48 isolated singular points.
  Let~$S$ be the minimal desingularization of~$\bar S$.
  The main result of this paper is the
  explicit determination of the Picard group of~$S$, including
  its structure as a Galois module over~$\Q$. The main ingredient for
  showing that the known subgroup is actually the full Picard group
  is the use of the combined action of the Galois group and the
  geometric automorphism group of~$S$ (which we also determine)
  on the Picard group. This reduces the proof to checking that the
  hyperplane section is not divisible by~2 in the Picard group.

  We use our explicit knowledge of the Picard group, together with that
  of a K3 surface obtained as a quotient of~$S$, to study curves of low
  degree on~$\bar{S}$. In this way, we completely classify all integral curves
  of degree at most~6 on~$\bar{S}$.
\end{abstract}

\maketitle


\section{Introduction}

Let $\BP^6$ be the projective space over $\Q$ with homogeneous coordinates
$a_1$, $a_2$, $a_3$, $b_1$, $b_2$, $b_3$, $c$; let $\bar{S}$ be the
surface in $\BP^6$ defined by
\begin{equation} \label{esse}
\left\{ \begin{array}{rrrrrrl}
            & a_1^2 &+& b_1^2 && = & c^2 \\[5pt]
            & a_2^2 &+& b_2^2 && = & c^2 \\[5pt]
            & a_3^2 &+& b_3^2 && = & c^2 \\[5pt]
          a_1^2 &+& a_2^2 &+& a_3^2 & = & c^2
        \end{array} \right.
\end{equation}
and note that the equations in~\eqref{esse} are equivalent to the equations
\begin{equation} \label{esse1}
\left\{ \begin{array}{rrrrrcl}
            a_1^2 &+& a_2^2 & &       & = & b_3^2 \\[5pt]
            a_1^2 & &       &+& a_3^2 & = & b_2^2 \\[5pt]
                  & & a_2^2 &+& a_3^2 & = & b_1^2 \\[5pt]
            a_1^2 &+& a_2^2 &+& a_3^2 & = & c^2.
        \end{array} \right.
\end{equation}

These equations encode the relations between the three sides
$a_1, a_2, a_3$, the three face diagonals $b_1, b_2, b_3$ and
the long diagonal~$c$ of a three-dimensional rectangular box.

The interest in this surface comes from a famous open problem:
\begin{center}
  \emph{Does there exist a `rational cuboid'?}
\end{center}
A rational cuboid is a (non-degenerate) rectangular box all of whose
sides, face diagonals and long diagonals have rational length.
The existence of a rational cuboid is therefore equivalent to the
existence of a rational point on~$\bar{S}$ with \hbox{$a_1 a_2 a_3 \neq 0$}.
See van~Luijk's undergraduate thesis~\cite{vanLuijk} for a summary
of the literature on this problem.

In this paper, we hope to make progress toward a better understanding
of the rational cuboid surface by proving some results on its geometry.
This extends results obtained in~\cite{vanLuijk}.

The surface~$\bar{S}$ has 48 isolated $A_1$~singularities. We let~$S$
denote the minimal desingularization of~$\bar{S}$. Then we show the
following.

\begin{theorem} \label{ThmAut}
  $\Aut_{\Qbar}(S) = \Aut_{\Qbar}(\bar{S})$ is an explicitly given group~$G$ of
  order 1536. Its action on~$\bar{S}$ extends to a linear action
  on the ambient~$\BP^6$.
\end{theorem}

This group comes from the obvious independent sign changes on all coordinates
and the equally obvious simultaneous permutations of the~$a$ and~$b$ coordinates,
together with a less obvious automorphism. This leads to an exact sequence
\[ 1 \To \mu_2^7/\mu_2 \To G \To \mathfrak{S}_4 \To 1 \,, \]
where the kernel comes from the sign changes
and the action of the $\mathfrak{S}_4$ quotient can be visualized
as the action of the symmetry group of the tetrahedron in the diagram below
(where an edge $\xymatrix@1{x \ar@{-}[rr]|{\textstyle y} & & z}$ corresponds to one
of the rank~$3$ quadrics $x^2 + z^2 = y^2$ defining~$\bar{S}$).
\begin{equation} \label{tetrahedron}
\begin{gathered}
   \xymatrix{ & & a_1 \ar@{-}[dddll]|{\textstyle b_3} \\
              \\
              & & ic \ar@{-}[uu]|{\textstyle ib_1}\ar@{-}[lld]|{\textstyle ib_2} \ar@{-}[rrd]|{\textstyle ib_3} \\
              a_2 \ar@{-}[rrrr]|{\textstyle b_1} & & & & a_3 \ar@{-}[uuull]|{\textstyle b_2}
            }
\end{gathered}
\end{equation}
See Section~\ref{SectAut} for details. The group was already known to van~Luijk;
we prove here that it is already the full automorphism group.

\begin{theorem} \label{ThmPic}
  The geometric Picard group of~$S$ has maximal rank
  \[ \rank \Pic S_{\Qbar} = \dim \H^{1,1}(S(\C)) = 64. \]
  It is generated by an explicitly
  known set of curves on~$\bar{S}$ together with the exceptional
  divisors. The discriminant of the intersection pairing on the
  Picard group is~$-2^{28}$.
\end{theorem}

We can take a suitable subset of~$64$ of the following curves (together with
the~$48$ exceptional divisors) as generators (see Definition~\ref{D:curves}).
\begin{enumerate}[$\bullet$]
  \item The 32 strict transforms of the conics in the four hyperplanes
        $a_1 = 0$, $a_2 = 0$, $a_3 = 0$, $c = 0$;
  \item the 12 strict transforms of the genus~1 curves contained in the three
        hyperplanes $b_1 = 0$, $b_2 = 0$, $b_3 = 0$;
  \item the 48 strict transforms of the genus~1 curves contained in the twelve
        hyperplanes $a_j = \eps a_{j+1}$ (where we set $a_4 := a_1$) or $a_j = \varepsilon i c$,
        where $j \in \{1,2,3\}$ and \hbox{$\eps \in \{1,-1\}$}.
\end{enumerate}

See Section~\ref{SectPic}. It is not hard to show that the geometric
Picard rank is~$64$, since one easily finds enough curves to generate
a group of that rank. These curves are already in~\cite{vanLuijk}.
The hard part is to show that the known curves
generate the full Picard group and not a proper subgroup of finite
index. Since~$2$ is the only prime number dividing the discriminant
of the known subgroup, it remains to show that no primitive element of
the known subgroup is divisible by~$2$. We use the known action of
the automorphism group together with the action of the absolute
Galois group of~$\Q$ to reduce the proof of saturation to the statement
that the single element corresponding to the hyperplane section is
not divisible by~$2$. This claim is then fairly easily established.

This technique, especially the arguments in the proof of Theorem~\ref{Thm2Gp},
may be helpful in similar situations, when one has a
fairly large group acting on the Picard group. Indeed, A.~V\'arilly-Alvarado
and B.~Viray~\cite{VarillyViray} use the technique described
in this paper\footnote{The work in this paper
predates~\cite{VarillyViray}; see~\cite{STarXiv}.} to compute the Picard
groups of various Enriques surfaces.

Several papers investigate restrictions that curves of geometric genus~0 or~1
on a surface must satisfy.
The work of Bogomolov in~\cite{Bo} highlighted the importance of symmetric
differentials.

In~\cite{BTV}, the authors study global sections of symmetric differentials
on certain surfaces of general type.
They establish what effect the singularities of the surface have on the sections
that they determine.
In turn, this allows them to deduce lower bounds for the number of singularities
that the curves must contain, as well as information about dimensions of the
linear spaces that the singularities must span.
In particular, when they apply their methods to the surface of
cuboids~\cite{BTV}*{Theorem~1.2}, they conclude that any rational curve,
other than the known conics, must pass through at least~7 nodes spanning
the ambient~$\mathbb{P}^6$.
See the corresponding paper for more details.

In~\cite{GFU}, the authors also study symmetric differentials, but, rather than
finding several explicit sections, they focus on solving the implied algebraic
differential equations that they impose on curves of genus at most~1.
This method was pioneered by Vojta~\cite{Vojta}.
When they apply their methods to the surface of cuboids~\cite{GFU}*{Theorem~1.2},
they conclude that all curves of genus at most~1 on~$\bar{S}$ must contain at
least~2 singular points.
Moreover, apart from the exceptional curves and the curves contained in
$a_1 a_2 a_3 c = 0$, every rational curve on~$\bar{S}$ must contain at
least~8 singular points (counted with multiplicity).
See the corresponding paper for more details.

We obtain a similar result with our methods (see Lemma~\ref{L:bounds}): Any
rational curve~$C$ on~$\bar{S}$ that is not a conic must satisfy $C \cdot E \ge 8$,
where~$E$ is the exceptional divisor (and we identify~$C$ with its strict transform),
and any curve~$C$ of geometric genus~$1$ on~$\bar{S}$ must satisfy $C \cdot E \ge 4$
(this gives an improvement over the result in~\cite{GFU}).

In~\cite{FM}, the authors study the surface of cuboids using theta functions.
They exploit the fact that the surface is a divisor in a Siegel modular variety
and determine explicitly a modular group associated to the surface.
In particular, they find an explicit product of modular curves that covers
the surface of cuboids.
This allows them to find bounds for the degree of a unibranch curve on the surface
in terms of the genus of the curve: see~\cite{FM}*{Theorem~3.1}.
For instance, unibranch rational curves on~$\bar{S}$ must have degree at most 176.
See the corresponding paper for more details.

\medskip

This paper contains a number of computational results. At~\cite{Code} we provide a
Magma~\cite{Magma} script that contains code that checks most of the computational
claims we make (set \texttt{quick := false} to also run checks that take more time).
We also provide a transcript of a Magma session where we have computed equations
for the fibrations given in Section~\ref{SectFib}.


\section{The Automorphism Group} \label{SectAut}

Throughout the paper, unless stated explicitly, all objects will be considered
over $\bar{\Q}$ (or~$\C$). In particular, $\Pic S$ denotes the geometric Picard
group and $\Aut(S)$ the geometric automorphism group of~$S$.

In this section, we determine the automorphism group of~$S$ and~$\bar{S}$.
We begin with some basic geometric properties of~$\bar{S}$.

\begin{lemma}
  The scheme $\bar{S}$ is a geometrically integral complete intersection of
  dimension two and multidegree $(2,2,2,2)$ in $\mathbb{P}^6$ with 48
  isolated $A_1$ singularities.
\end{lemma}

\begin{proof}
  Every irreducible component of $\bar{S}$ has dimension at least two, since $\bar{S}$
  is defined by four equations.  If $\bar{S}$ had a component of dimension at least
  three, then the intersection of $\bar{S}$ with any hyperplane would have a component
  of dimension at least two.  Let~$C$ be the hyperplane defined by the vanishing of $c$.
  The scheme $\bar{S} \cap C$ is the union of eight smooth conics defined over the field
  $\Q(i)$ by
  \[ c = 0, \quad b_1 = \eps_1 i a_1, \quad  b_2 = \eps_2 i a_2, \quad
     b_3 = \eps_3 i a_3, \quad a_1^2 + a_2^2 + a_3^2 = 0
  \]
  for $\eps_1, \eps_2, \eps_3 \in \{\pm1\}$.
  In particular it is pure of dimension one, so that every irreducible component of
  $\bar{S}$ has dimension two; note also that $\bar{S} \cap C$ is reduced.
  It follows that $\bar{S}$ is the complete intersection of the equations
  in~\eqref{esse}.

  Thus $\bar{S}$ is Cohen-Macaulay of dimension two, and to prove that it is integral
  it suffices to show that the singular locus of $\bar{S}$ has codimension at
  least two.
  Since $\bar{S} \cap C$ is reduced, no component of dimension one of the singular
  locus of $\bar{S}$ is contained in $C$.
  Let $p = [\alpha _1 , \alpha _2 , \alpha _3 , \beta _1 ,
            \beta _2, \beta _3, \gamma ]$ be a point in $\bar{S} \setminus C$.
  Examining the
  equations of $\bar{S}$ we see immediately that if $\alpha_1 \alpha_2 \alpha_3 \neq 0$
  or $\beta_1 \beta_2 \beta_3 \neq 0$, then the rank of the Jacobian of the equations
  at~$p$ is four and such points are therefore smooth.
  We conclude at once that the singular
  points of $\bar{S}$ are the points for which all three coordinates appearing
  in one of the six rank three quadrics in~\eqref{esse} or~\eqref{esse1} vanish; in
  particular, the surface $\bar{S}$ has only finitely many singular points.
  The fact that the singular points are 48 and that they are of type $A_1$ is
  immediate from the equations.
\end{proof}

As a corollary, we deduce that $\bar{S}$ is Cohen-Macaulay, Gorenstein, reduced,
normal and projectively normal.  By the adjunction formula, the canonical sheaf
on $\bar{S}$ is the sheaf $\cO_{\bar{S}}(1)$.  Since $\bar{S}$ is projectively normal, we
deduce also that $p_g(\bar{S}) = 7$; it is an easy calculation to see that
$\chi \bigl(\bar{S}, \cO_{\bar{S}}\bigr) = 8$. Let $b \colon S \to \bar{S}$ be the blow-up
of $\bar{S}$ at its 48 singular points; thus~$S$ is smooth and it is the minimal
desingularization of $\bar{S}$.  Denote by $K_S$ a canonical divisor of~$S$; we
have \hbox{$\cO_S(K_S) \simeq b^* \cO_{\bar{S}}(1)$} and \hbox{$(K_S)^2 = 16$}.  Since the
singularities of $\bar{S}$ are rational double points, we have
$\chi \bigl(S, \cO_{S}\bigr) = \chi \bigl(\bar{S}, \cO_{\bar{S}}\bigr) = 8$ and also
$p_g(S) = 7$ and $q(S) = 0$. Using Noether's formula we finally deduce that the Hodge
diamond of~$S$ is
\[ \begin{array}{cccccc}
        &   & 1 \\
        & 0 &    & 0 \\
      7 &   & 64 &   & 7 \\
        & 0 &    & 0 \\
        &   & 1
   \end{array}
\]
It follows from the above that the rational map associated to the canonical divisor
$K_S$ on~$S$ is a morphism and that it is the contraction of the 48 exceptional curves
of~$b$, followed by the inclusion of~$\bar{S}$ into~$\BP^6$.  Therefore the canonical
divisor on~$S$ is big and nef, so that~$S$ is a minimal surface of general type
and~$\bar{S}$ is its canonical model: there are no curves on~$S$ with negative
intersection with the canonical divisor, and, in particular, there are no $(-1)$-curves
on~$S$.  Moreover, the 48 exceptional curves of~$b$ are the only $(-2)$-curves
on~$S$.  Since the morphism $S \to \mathbb{P}^6$ is the morphism associated to the
canonical divisor, the automorphism
groups of~$S$ and~$\bar{S}$ coincide; denote this group by~$G$.  The group~$G$ is
naturally identified with the subgroup of $\Aut(\BP^6)$ preserving the
subscheme~$\bar{S}$, since the canonical divisor class of~$S$ is $G$-invariant.

The symmetric group $\mathfrak{S}_3$ on~$\{1, 2, 3\}$ acts on~$\BP^6$
and~$\bar{S}$ by permuting simultaneously the indices of $a_1,a_2,a_3$ and $b_1,b_2,b_3$
and fixing~$c$.  Note that also the linear automorphism of order two
\[ \sigma \colon \left\{ \begin{array}{l@{~\mapsto~}l@{\hspace{20pt}}l@{~\mapsto~}l}
                      a_1 & a_1 & b_1 & -ib_2 \\
                      a_2 & a_2 & b_2 & ib_1  \\
                      a_3 & -ic & b_3 & b_3   \\
                      c   & ia_3
            \end{array} \right.
\]
of~$\BP^6$ preserves $\bar{S}$ and therefore induces an automorphism of~$S$.
Let $G' \subset G$ be the subgroup generated by $\mathfrak{S}_3$, $\sigma$ and all
the sign changes of the variables.

\begin{prop}
  The group $G = \Aut(S)$ is equal to $G'$.
\end{prop}

\begin{proof}
  First, we show that the rank three quadrics vanishing on~$\bar{S}$ are the six rank
  three quadrics appearing in~\eqref{esse} and~\eqref{esse1}.
  For $j \in \{1,2,3\}$, let
  \[ q_j := a_j^2 + b_j^2 - c^2 \quad\text{and}\quad
     r_j := a_1^2 + a_2^2 + a_3^2 - a_j^2 - b_j^2,
  \]
  and let $Q_j := V(q_j)$ and $R_j := V(r_j)$. Let
  \[ q_4 := a_1^2 + a_2^2 + a_3^2 - c^2, \]
  and let~$q$ be an equation of a quadric vanishing on~$\bar{S}$.
  Since the ideal of~$\bar{S}$ is generated by the equations in~\eqref{esse}, we have
  $q = \lambda_1 q_1 + \lambda_2 q_2 + \lambda_3 q_3 + \lambda_4 q_4$, for some
  $\lambda_1, \lambda_2, \lambda_3, \lambda_4 \in \Qbar$.

  If at least two of the coefficients $\lambda_1, \lambda_2, \lambda_3$ are non-zero,
  then we easily see that~$q$ has rank at least four.  Thus at most one of the
  coefficients $\lambda_1, \lambda_2, \lambda_3$ is non-zero, and we conclude that
  the rank three quadrics vanishing on $\bar{S}$ are $q_1, q_2, q_3, r_1, r_2, r_3$.
  It follows that the set $\cQ := \{Q_1, Q_2, Q_3, R_1, R_2, R_3\}$ is fixed by the
  induced action of~$G$.

  Second, we show that $G'$ acts transitively on~$\cQ$.  This is immediate: the group
  $\mathfrak{S}_3$ acts transitively on $\{Q_1, Q_2, Q_3\}$ and~$\{R_1, R_2, R_3\}$,
  while $\sigma $ acts transitively on $\{Q_1, R_2\}$ and~$\{Q_2, R_3\}$ (and it fixes
  $Q_3$ and~$R_1$).

  Third, we show that the stabilizer of~$Q_1$ in~$G$ is the same as the stabilizer of
  $Q_1$ in~$G'$; from this the result follows since~$G$ and~$G'$ act transitively
  on~$\cQ$.  Let $\tau \in G$ be an automorphism of~$\BP^6$ stabilizing $\bar{S}$
  and~$Q_1$.  In particular $\tau$ stabilizes $Q_1^{\sing}$, the singular locus
  of~$Q_1$, and $Q_1^{\sing} \cap \bar{S}$.  The (underlying set of the) intersection
  $Q_1^{\sing} \cap \bar{S}$ is the set of eight points
  $\bigl\{[0, 1, \eps_1, 0, \eps_2 i, \eps_3 i, 0] :
             \eps_1, \eps_2, \eps_3 \in \{1,-1\} \bigr\}$.
  Clearly the actions of the stabilizers in~$G$ and~$G'$ are transitive on this set,
  since both groups contain arbitrary sign changes of the variables, and if an
  automorphism of~$V(a_1,b_1,c) \simeq \BP^3$ fixes the eight points above,
  then it is the identity (the points for which
  $\eps_1 \eps_2 \eps_3 = 1$ are essentially the characters of
  $(\mathbb{Z}/2\mathbb{Z})^2$, and therefore the corresponding points span
  $\mathbb{P}^3$).  Thus, we may assume that $\tau$ acts as the identity on
  $V(a_1,b_1,c) \subset \mathbb{P}^6$.  Hence $\tau$ fixes the variable~$c$, up to a
  sign, since it fixes $a_2, b_2$;
  similarly it fixes $a_1$ and~$b_1$ up to signs, since
  it fixes $a_2, b_3$ and $a_2, a_3$ respectively, and we conclude that
  $G' = \Aut(\bar{S}) = \Aut(S)$.
\end{proof}

To compute the size of the group $G$, note that each element of $G = G'$
induces a permutation of the set
$\{a_1^2, a_2^2, a_3^2, -c^2\}$ and every permutation is obtained.
This gives a surjective group homomorphism $G \to \mathfrak{S}_4$.
The subgroup $\langle \mathfrak{S}_3, \sigma \rangle$ permutes
$\{a_1, a_2, a_3, -ic\}$. One can check that the kernel of
$\langle \mathfrak{S}_3, \sigma \rangle \to \mathfrak{S}_4$
consists of maps that fix $b_1^2, b_2^2, b_3^2$. So the kernel
of $G \to \mathfrak{S}_4$ consists exactly of the automorphisms that
change the signs of a subset of the variables. This shows that
\[ \#G = 2^{7-1} \#\mathfrak{S}_4 = 64 \cdot 24 = 1536 . \]
The group~$G'$ is described in~\cite{vanLuijk}*{p.~25} as a subgroup
of the automorphism group of~$\bar{S}$. The new statement here is
that it is already the full automorphism group.


\section{The Picard Group} \label{SectPic}

In this section, we determine the (geometric) Picard group of~$S$.
We first show that the hyperplane section of~$\bar{S}$ is not divisible
in the Picard group.  Recall that the canonical class of~$S$ is the
pull-back of the hyperplane class of~$\bar{S}$.

\begin{lemma} \label{dueca}
  The canonical divisor class of~$S$ is a primitive vector in $\Pic S$.
\end{lemma}

\begin{proof}
  Let $K_S$ be a canonical divisor of~$S$; since $(K_S)^2 = 16$, it suffices to show
  that there is no divisor~$R$ on~$S$ such that $K_S \sim 2 R$.  We argue by
  contradiction and suppose that such a divisor~$R$ exists.  By the Riemann-Roch
  formula we deduce that
  \[ h^0(S,\cO_S(R)) + h^2(S,\cO_S(R)) \geq 6, \]
  and from Serre duality it follows that $h^2(S,\cO_S(R)) = h^0(S,\cO_S(R))$;
  thus
  \[ h^0(S,\cO_S(R)) \geq 3. \]
  Note also that $2 \dim |R| \leq \dim |2R| = 6$, and therefore
  $h^0(S,\cO_S(R)) \in \{3,4\}$.
  The image $S'$ of~$S$ under the rational map determined
  by the linear system~$|R|$ is an irreducible, non-degenerate subvariety
  of~$|R|^\vee$.
  Let~$k$ be the number of independent quadratic equations vanishing along~$S'$.
  The image of
  $\Sym^2 \H^0(S,\cO_S(R))$ in $\H^0(S, \cO_S(K_S))$
  is an $\Aut(S)$-invariant subspace. We easily see that the non-trivial
  $\Aut(S)$-invariant subspaces of $\BP^6$ are $V(a_1,a_2,a_3,c)$
  and~$V(b_1,b_2,b_3)$.  It follows that the image of
  $\Sym^2 \H^0(S,\cO_S(R))$ in $\H^0(S, \cO_S(K_S))$ has dimension
  \[ \binom{h^0(S,\cO_S(R))+1}{2} - k \]
  and that this number must be in $\{0,3,4,7\}$.

  \noindent
  {\bf Case 1:} $h^0(S,\cO_S(R)) = 3$, so that $S' \subset \mathbb{P}^2$. \\
  Since $S'$ is irreducible and non-degenerate, we deduce that $S'$ cannot be
  contained in two independent quadrics, or it would be degenerate.
  Thus $S'$ must be defined by $k \le 1$ quadrics, and this is incompatible
  with the previous constraints.

  \noindent
  {\bf Case 2:} $h^0(S,\cO_S(R)) = 4$, so that $S' \subset \mathbb{P}^3$. \\
  Thus $k \ge 3$.  Since $S'$ is irreducible and non-degenerate, any quadric
  vanishing on $S'$ must be irreducible. Hence, any two independent quadrics
  vanishing along $S'$ intersect in a scheme of pure dimension one and degree four.
  Since $S'$ is non-degenerate, its degree is at least three.
  Because there are at least three independent quadrics vanishing on $S'$,
  this implies $k=3$, and $S'$ is a twisted cubic curve. But then the map
  $\Sym^2 \H^0(S,\cO_S(R)) \to \H^0(S, \cO_S(K_S))$
  is surjective, which implies that the image of~$S$ under~$K_S$
  factors through the image under~$R$.
  So $\bar{S}$ would have to be a curve, contradicting the fact
  that it is a surface.

  Thus the class of the canonical divisor $K_S$ is not the double of a divisor
  class on~$S$, and the proof is complete.
\end{proof}

Next we prove that $\Pic S$ is a free abelian group of rank~64
and find a set of generators of a subgroup of finite index.

\begin{definition} \label{D:curves}
  We list some sets of curves on~$S$ (we let $a_4 := a_1$, to simplify the notation).
  \begin{enumerate}[(1)]
    \item $\cG_0$ --- the 48 exceptional curves of the resolution $S \to \bar{S}$
          (24 defined over $\Q$, 24 defined over $\Q(i)$);

    \item $\cG_1$ --- the 32 strict transforms of the conics in the four hyperplanes
          $a_1 = 0$, $a_2 = 0$, $a_3 = 0$, $c = 0$
          (24 defined over $\Q$, 8 defined over $\Q(i)$);

    \item $\cG_2$ --- the 12 strict transforms of the genus~1 curves contained in the three
          hyperplanes $b_1 = 0$, $b_2 = 0$, $b_3 = 0$ (defined over $\Q(i)$);

    \item $\cG_3$ --- the 48 strict transforms of the genus~1 curves contained in the twelve
          hyperplanes $a_j = \eps a_{j+1}$ or $a_j = \varepsilon i c$,
          where $j \in \{1,2,3\}$ and $\eps \in \{1,-1\}$
          (24~defined over $\Q(\sqrt{2})$, 24 defined over $\Q(i,\sqrt{2})$).
  \end{enumerate}
\end{definition}

These curves are already described in~\cite{vanLuijk}*{p.~48}. Denote by
\[ \cG = \cG_0 \cup \cG_1 \cup \cG_2 \cup \cG_3 \]
the set consisting of the above 140 curves.

\begin{prop} \label{PropRank}
  The geometric Picard group of~$S$ is a free abelian group of rank~64, and
  the curves in~$\cG$ generate a subgroup of finite index.
\end{prop}

\begin{proof}
  Since $h^1(S, \cO_S) = 0$, the Picard group is finitely generated.
  Moreover, each torsion class in the Picard group determines an
  unramified connected cyclic covering $\pi \colon \tilde S \to S$ that is
  trivial if and only if the class in the Picard group is trivial.  Any
  such cover induces a similar cover on~$\bar{S}$: the inverse image
  under~$\pi$ of each $(-2)$-curve in~$S$ consists of a disjoint union
  of $(-2)$-curves in~$\tilde S$, which can be contracted to rational double
  points, thus obtaining an unramified connected cyclic cover of~$\bar{S}$.
  By \cite{Di}*{Thm.~2.1} we have $\H_1(\bar{S}, \Z) = 0$, and hence the cyclic
  covering is trivial. We obtain that all torsion classes in $\Pic S$ are trivial,
  and $\Pic S$ is torsion free.

  From the calculation of the Hodge numbers of~$S$, we deduce that the rank
  of $\Pic S$ is at most~64.  Thus, to conclude, it suffices to show that the
  intersection matrix of the curves in~$\cG$ has rank~64.  By the adjunction
  formula, we easily see that the self-intersection of the divisor class of
  a conic or a genus one curve in our list is~$-4$ (in both cases, we mean the
  strict transform
  in~$S$ of the corresponding curve).  The evaluation of the remaining pairwise
  intersection numbers of the curves above is straightforward, tedious, and
  preferably done by computer.  We check using a computer that the rank of the
  intersection matrix of the 140~curves in~$\cG$ is~64, concluding the proof of
  the proposition.
\end{proof}

\begin{theorem} \label{Thm2Gp}
  The Picard group of~$S$ is a free abelian group of rank~64,
  and it is generated by the classes of curves in $\cG$.

  The discriminant of the intersection pairing on~$\Pic S$ is~$-2^{28}$.
\end{theorem}

\begin{proof}
  By Proposition~\ref{PropRank}, the Picard group is free abelian of
  rank~64, and the lattice~$L$ generated by the classes of the curves in~$\cG$
  is of finite index in the Picard group. It remains to show that $L$
  is already the full Picard group.

  The computation of the intersection
  matrix shows that the discriminant of~$L$ is~$-2^{28}$.
  Thus the cokernel of the inclusion $L \to \Pic S$ is a finite abelian 2-group.
  Hence, to prove the equality $L = \Pic S$ it suffices to prove that the natural
  morphism $L/2L \to \Pic S/2 \Pic S$ is injective;
  denote by $L_2$ its kernel.
  The Galois group $\Gal(\Q(i,\sqrt{2})/\Q)$ acts on the lattice~$\Pic S$,
  since the divisors in~$\cG$ are defined over~$\Q(i,\sqrt{2})$, while~$S$
  is defined over~$\Q$.  Denote by $\tilde G$ the group of automorphisms
  of~$\Pic S$ generated by~$G$ and~$\Gal(\Q(i,\sqrt{2})/\Q)$.
  The action of $\tilde G$ on~$\Pic S$ induces an action of $\tilde G$ on~$L$,
  since $\cG$ is stable under the action of both $G$
  and~$\Gal(\Q(i,\sqrt{2})/\Q)$, and hence there is an action of~$\tilde G$ on~$L/2L$.
  The $\F_2$-vector space $L_2 \subset L/2L$ is invariant under $\tilde G$
  (it is the kernel of a $\tilde G$-equivariant homomorphism).
  Let $\tilde G_2 \subset \tilde G$ be a Sylow 2-subgroup.
  Any representation of a 2-group on an $\F_2$-vector space of positive dimension
  has a non-trivial fixed subspace.
  In particular, there is a non-trivial $\tilde G_2$-invariant subspace
  in~$L_2$.  Using a computer we check that the subspace of $(L/2L)^{\tilde G_2}$
  coming from classes in~$L$ that have even intersection with~$L$ has dimension~1,
  spanned by the reduction modulo~2 of the canonical class (note that the canonical
  class is fixed by the action of~$\tilde G$ on~$\Pic S$).
  We deduce that if $L_2 \neq 0$, then the canonical divisor class is divisible
  by~two in $\Pic S$.  By Lemma~\ref{dueca} we know that this is not the case.
  It follows that $\Pic S$ is generated by the classes of the curves in~$\cG$.
\end{proof}

\begin{cor}
  The intersection pairing on $\Pic S$ is even; in particular, there are no curves of
  odd degree on the surface~$S$.
\end{cor}

\begin{proof}
  By Theorem~\ref{Thm2Gp} the Picard group of~$S$ is generated by the elements
  of~$\cG$; since all the elements of~$\cG$ have even self-intersection, we deduce
  that the pairing on $\Pic S$ is even.  By the adjunction formula, the degree of any
  curve on~$S$ has the same parity as its self-intersection; since the pairing on
  $\Pic S$ is even, we conclude that every curve on~$S$ has even degree.
  (This latter statement would also follow from the fact that all curves in~$\cG$
  have even degree.)
\end{proof}

As a consequence, also the surface~$\bar{S}$ contains no curves of odd degree; van~Luijk had
already shown that there are no lines on~$\bar{S}$, see~\cite{vanLuijk}*{Prop.~3.4.11}.

Using the explicitly known structure of $\Pic S$ as a Galois module, we obtain
the following result.

\begin{theorem}
  The algebraic part of the Brauer group of~$S$ is the isomorphic image
  of the Brauer group of~$\Q$ in the Brauer group of~$S$.
\end{theorem}

\begin{proof}
  It is well-known that the cokernel of the inclusion of the Brauer group of~$\Q$
  in the algebraic part of the Brauer group of~$S$ is isomorphic to the Galois
  cohomology group $\H^1(\Q, \Pic S)$. Since $\Pic S$ is torsion free and
  we found a set of generators of the Picard group of~$S$ defined
  over $\Q(i, \sqrt{2})$, we have
  \[ \H^1(\Q, \Pic S) = \H^1(\Gal(\Q(i, \sqrt{2})/\Q), \Pic S) . \]
  A computation in Magma~\cite{Magma} shows that the latter cohomology group vanishes,
  establishing the result.
\end{proof}

This means that there is no `algebraic Brauer-Manin obstruction' to
weak approximation on~$S$. It would be interesting to investigate the
transcendental quotient of the Brauer group.


\section{The cuboid surface as a modular surface} \label{SectMod}

In the paper~\cite{Beau}, Beauville gives a construction of~$\bar{S}$
as a quotient of a product of curves. We consider the curve~$X$ in~$\BP^4$ given
by the equations
\[ u^2 = 2 x y, \quad v^2 = x^2 - y^2, \quad w^2 = x^2 + y^2 \,. \]
This is a smooth curve of genus~5, which is a model of the modular curve~$X(8)$
whose non-cuspidal points correspond to elliptic curves~$E$ together with a
symplectic isomorphism $\Z/8\Z \times \mu_8 \to E[8]$. The geometric automorphism
group of~$X$ is $\PSL(2,\Z/8\Z)$, a group of order~$192 = 3 \cdot 2^6$.
The canonical epimorphism $\PSL(2, \Z/8\Z) \to \PSL(2, \Z/4\Z) \cong \mathfrak{S}_4$
has kernel~$G_0$ isomorphic to $(\Z/2\Z)^3$; in terms of our model, the automorphisms
in the kernel are given by sign changes of~$u$, $v$ and~$w$.

We consider the diagonal action of~$G_0$ on $X \times X$. Denoting the
coordinates on the first factor by $u_1, v_1, w_1, x_1, y_1$ and on the second
factor by $u_2, v_2, w_2, x_2, y_2$, invariants of the action are given by
$U = u_1 u_2$, $V = v_1 v_2$, $W = w_1 w_2$, $X = x_1 x_2$, $Y = y_1 y_2$
and $T = x_1 y_2$, $Z = x_2 y_1$. We obtain the relations
\[ X Y = T Z, \quad U^2 = 4 X Y, \quad V^2 = X^2 + Y^2 - T^2 - Z^2, \quad
   W^2 = X^2 + Y^2 + T^2 + Z^2 \,.
\]
Setting
\begin{gather*}
  U = 2 b_1, \quad V = 2 b_2, \quad W = 2 b_3, \\
  X = a_1 + c, \quad Y = -a_1 + c, \quad T = a_2 + i a_3, \quad Z = a_2 - i a_3 \,,
\end{gather*}
this gives the equations
\begin{align*}
  a_1^2 + a_2^2 + a_3^2 &= c^2 \\
  a_1^2 + b_1^2 &= c^2 \\
  a_1^2 - a_2^2 + a_3^2 + c^2 &= 2 b_2^2 \\
  a_1^2 + a_2^2 - a_3^2 + c^2 &= 2 b_3^2 \,,
\end{align*}
which are equivalent to the equations defining~$\bar{S}$. Note that the isomorphism
\hbox{$(X \times X)/G_0 \stackrel{\cong}{\to} \bar{S}$} is defined over~$\Q(i)$
and not over~$\Q$.

We have $X/G_0 \cong X(4) \cong \BP^1$. The action of $(G_0 \times G_0)/G_0 \cong G_0$
(with $G_0$ embedded diagonally) on the quotient $(X \times X)/G_0 \cong \bar{S}$
is via sign changes on $u_1 u_2 = 2 b_1$, $v_1 v_2 = 2 b_2$ and $w_1 w_2 = 2 b_3$.
Therefore the quotient $\bar{S}/G_0$ is obtained by projecting away from the linear
subspace spanned by these coordinates; it is the quadric
\[ Q \colon a_1^2 + a_2^2 + a_3^2 = c^2 \]
in~$\BP^3$, which splits over~$\Q(i)$, so that $Q \cong \BP^1 \times \BP^1$.

A point on $X \times X = X(8) \times X(8)$ corresponds to a pair of elliptic
curves with full level-8 structure. Since~$\bar{S}$ maps to $X(4) \times X(4)$,
the points on~$\bar{S}$ give rise to a pair of elliptic curves with full level-4
structure; dividing by the diagonal action of~$G_0$ corresponds to only keeping
the induced isomorphism between the 8-torsion subgroups. Therefore, a point on~$\bar{S}$
corresponds to a pair of elliptic curves~$E$ and~$E'$ with isomorphisms $\phi$
and~$\psi$ as in the following diagram:
\[ \Z/4\Z \times \mu_4 \stackrel{\phi}{\to} E[4] \subset E[8] \stackrel{\psi}{\to} E'[8] \,. \]
This holds for all fields containing a square root of~$-1$. To find out what the
correct moduli problem is over~$\Q$, we observe that the sign change $a_3 \mapsto -a_3$
on~$\bar{S}$ lifts to the automorphism of $X \times X$ that switches the two factors.
This implies that over~$\Q$,~$\bar{S}$ is isomorphic to the quotient of
$R_{\Q(i)/\Q} X_{\Q(i)}$, the Weil restriction of scalars down to~$\Q$ of $X$
base-changed to~$\Q(i)$, by~$G_0$. The quadric~$Q$ is then $R_{\Q(i)/\Q}(X(4)_{\Q(i)})$.
A rational point on~$\bar{S}$ therefore corresponds to an elliptic curve~$E$ over~$\Q(i)$
with a basis $P_1, P_2$ of~$E[4]$ such that there is an isomorphism
$\psi \colon E[8] \to \bar{E}[8]$ with $\psi(P_1) = \bar{P}_1$ and $\psi(P_2) = -\bar{P}_2$,
where the bar denotes the action of the nontrivial automorphism of~$\Q(i)$.

Let $G_0^+$ be the subgroup of~$G_0$ whose elements perform sign changes on
an even number of variables. Then $X/G_0^+$ is the genus~2 curve
\[ C \colon y^2 = 2(x^5 - x) \,. \]
The quotient $Y = (X \times X)/G_0^+$ is a smooth surface that is a double cover
of~$\bar{S}$, ramified exactly in the 48~singularities of~$\bar{S}$. On the other
hand, $Z = \bar{S}/G_0^+$ has $C \times C$ as a double cover (over~$\Q(i)$).
Over~$\Q$ we find $R_{\Q(i)/\Q} C_{\Q(i)}$ as a double cover of~$Z$.
The surface~$Z$ in turn is a double cover of the quadric~$Q$, branched over
a divisor of type~$(6,6)$, so~$Z$ is still of general type.


\section{Fibrations in curves of genus 5} \label{SectFib}

We have observed earlier that there are exactly six quadrics of rank~3 that
contain~$\bar{S}$. In a similar way, one easily sees that there are exactly
eleven quadrics of rank~4 containing~$\bar{S}$. These are diagonal quadrics
involving all possible subsets of size~4 of the variables that do not contain
one of the sets of three variables corresponding to the rank~3 quadrics.
Explicitly, the eleven quadrics are given by
{\renewcommand{\arraystretch}{1.2}
\[ \begin{array}{*{13}{@{\,\,}c}@{{}=0}}
      a_1^2 &+& a_2^2 &+& a_3^2 & &       & &       & &       &-& c^2 \\
      a_1^2 &-& a_2^2 & &       &+& b_1^2 &-& b_2^2 & &       & &     \\
            & &       & & a_3^2 &-& b_1^2 &-& b_2^2 & &       &+& c^2 \\
            & & a_2^2 &-& a_3^2 & &       &+& b_2^2 &-& b_3^2 & &     \\
      a_1^2 & &       & &       & &       &-& b_2^2 &-& b_3^2 &+& c^2 \\
     -a_1^2 & &       &+& a_3^2 &-& b_1^2 & &       &+& b_3^2 & &     \\
            & & a_2^2 & &       &-& b_1^2 & &       &-& b_3^2 &+& c^2 \\
     2a_1^2 & &       & &       &+& b_1^2 &-& b_2^2 &-& b_3^2 & &     \\
           & & 2a_2^2 & &       &-& b_1^2 &+& b_2^2 &-& b_3^2 & &     \\
           & &       & & 2a_3^2 &-& b_1^2 &-& b_2^2 &+& b_3^2 & &     \\
            & &       & &       & & b_1^2 &+& b_2^2 &+& b_3^2 &-& 2c^2
   \end{array}
\]
}%
They can be associated in the given order to the tetrahedron in~\eqref{tetrahedron},
its six edges, and its four vertices. This partition corresponds to the orbits under the
automorphism group~$G$.

Projecting away from the space spanned by the remaining variables, each of
these gives rise to a rational map from~$\bar{S}$ onto a quadric in~$\BP^3$.
Since~$\bar{S}$ does not meet the first and the last four of these planes,
the corresponding rational maps are morphisms.
The first of these is the morphism $\bar{S} \to Q$, where~$Q$ is the quadric
mentioned in Section~\ref{SectMod}. On the other hand,~$\bar{S}$ does meet
the planes corresponding to the remaining six rank~4 quadrics (in eight singular
points each), so to obtain a morphism, one has to blow up these points.

Over a field that splits the quadric, we can post-compose with one of the projections
to~$\BP^1$. The fibers of the resulting map $\bar{S} \to \BP^1$ are curves
of arithmetic genus~5 and degree~8; the generic fiber is a smooth irreducible
curve canonically embedded into the~$\BP^4$ it spans. In this way, each of the
eleven rank~4 quadrics gives rise to two complementary (in the sense that the
sum of the classes of their fibers is the hyperplane section) fibrations of~$\bar{S}$
in curves of genus~5. The first pair of these, which are defined over~$\Q(i)$,
are isotrivial and correspond to the fibrations induced from the product
$X \times X$ in Section~\ref{SectMod}. The other fibrations are not isotrivial.
(We can check this by considering one of the genus~1 quotient fibrations
and verifying that the $j$-invariant is not constant.)
Since the quadrics are invariant under the
sign changes of the three variables not occurring in them, the corresponding
group $(\Z/2\Z)^3$ acts on all the fibers. Quotienting out by subgroups,
we obtain surfaces fibered in hyperelliptic (two simultaneous sign changes)
or non-hyperelliptic (simultaneous sign change of all three variables) curves
of genus~3 or in curves of genus~2 (even number of sign changes).
We also find quotients fibered in genus~1 curves (one sign change).
Dividing by the full $(\Z/2\Z)^3$, we get~$\bar{S}$ as a Galois cover of
$\BP^1 \times \BP^1$ (at least geometrically). For the first quadric, the
genus~2 fibration is~$Z$ in the notation of Section~\ref{SectMod}; we saw
there that this quotient is of general type, which implies that the hyperelliptic
genus~3 fibration quotients covering it are of general type as well.

There is a similar construction starting from a rank~3 quadric. Projecting away
from the space spanned by the four variables not occurring in the quadric,
we map~$\bar{S}$ onto a conic. Here we have eight (singular) points of~$\bar{S}$
in the base locus, so to get a morphism, we have to blow them up. Post-composing
with an isomorphism between the conic and~$\BP^1$ (this is always possible
already over~$\Q$), we obtain maps $\bar{S} \to \BP^1$ again. The fibers (in~$\bar{S}$)
are again generically smooth canonical curves of genus~5, but this time, we
get only one fibration from each rank~3 quadric. Here, twice the class of the
fiber (on~$S$) is the hyperplane section minus the eight exceptional curves coming
from blowing up the singularities in the base locus.
The genus~5 curves occurring as fibers are given by diagonal quadrics; we get
an action of~$(\Z/2\Z)^4$ on them. This leads to more quotients fibered into
curves of genus 1, 2 or~3.

In total, we obtain $6 + 2 \cdot 11 = 28$ such fibrations.
We summarize these results in the following proposition.

\begin{prop}[Genus 5 fibrations]\label{prop:genus5fibrations}
The surface $\bar S$ admits exactly 28 fibrations in curves of arithmetic genus~5, arising from the six rank--3 quadrics and the eleven rank--4 quadrics containing $\bar S$, as detailed above.\end{prop}

The (isotrivial) fibrations coming from the first rank~$4$ quadric above have six bad
fibers each (for the fibration given by $[a_1 + i a_2, a_3 + c]$, they are
above the points $0, \infty, \pm 1, \pm i$), which each consist of one of the genus~1
curves in~$\cG_2$ taken twice. The 12~curves in~$\cG_2$ each occur
exactly once in this way.

The fibrations coming from the next six quadrics of rank~4 also have six
bad fibers each (again over the same points, for $[a_1 + a_2, b_1 + b_2]$, say).
Two of them consist of two of the genus~1 curves in~$\cG_3$ joined by four of the
exceptional curves,
the other four bad fibers consist of two conics from~$\cG_1$, each with multiplicity~2
and joined by two exceptional curves. The curves in~$\cG_3$ show up exactly once
in this way, whereas each conic in~$\cG_1$ occurs in three of the fibrations.

The fibrations coming from the last four quadrics of rank~4 have six fibers each
splitting into two curves from~$\cG_3$ as above and twelve fibers that are hyperelliptic
curves of genus~3 with two nodes (at singular points of~$\bar{S}$). These curves
form an orbit of size~96 under~$\Aut(S)$ and are birational to the curve
\[ y^2 = 9 x^8 + 20 x^6 + 86 x^4 + 20 x^2 + 9 \,. \]

The fibrations coming from the six quadrics of rank~3 each have four fibers splitting
into four conics (from~$\cG_1$), four fibers splitting into two genus~1 curves
from~$\cG_3$ and two fibers splitting into two genus~1 curves from~$\cG_2$.

\medskip

We list equations for the fibers of some representatives of the four different
types of fibrations described above.

\begin{itemize}
  \item The rank~$3$ quadric $a_1^2 + b_1^2 = c^2$ with $t = (a_1 + c)/b_1$:
        \begin{align*}
          (t^2 + 1) a_1 - (t^2 - 1) c &= 0 \\
          (t^2 + 1) b_1 - 2 t c &= 0 \\
          (t^2 + 1)^2 a_2^2 - (t^2 + 1)^2 b_3^2 + (t^2 - 1)^2 c^2 &= 0 \\
          a_3^2 + b_3^2 - c^2 &= 0 \\
          (t^2 + 1)^2 b_2^2 + (t^2 + 1)^2 b_3^2 - 2 (t^4 + 1) c^2 &= 0
        \end{align*}
  \item One of the two fibrations associated to $a_1^2 + a_2^2 + a_3^2 = c^2$ \\
        with $t = (c + a_1)/(a_2 + i a_3) = (a_2 - i a_3)/(c - a_1)$:
        \begin{align*}
          (t^2 + 1) a_1 - 2 i t a_3 - (t^2 - 1) c &= 0 \\
          (t^2 + 1) a_2 + i (t^2 - 1) a_3 - 2 t c &= 0 \\
          a_3^2 + b_3^2 - c^2 &= 0 \\
          (t^2 + 1)^2 b_1^2 + 4 t^2 b_3^2 + 4 i t (t^2 - 1) a_3 c - 8 t^2 c^2 &= 0 \\
          (t^2 + 1)^2 b_2^2 + (t^2 - 1)^2 b_3^2 - 4 i t (t^2 - 1) a_3 c - 2 (t^2 - 1)^2 c^2 &= 0
        \end{align*}
  \item One of the two fibrations associated to $a_1^2 - a_2^2 + b_1^2 - b_2^2 = 0$ \\
        with $t = (a_1 + a_2)/(b_2 + b_1) = (b_2 - b_1)/(a_1 - a_2)$:
        \begin{align*}
          2 t a_1 - (t^2 - 1) b_1 - (t^2 + 1) b_2 &= 0 \\
          2 t a_2 - (t^2 + 1) b_1 - (t^2 - 1) b_2 &= 0 \\
          2 (t^4 - 1) b_1 b_2 - (t^2 - 1)^2 b_3^2 + 2 (t^4 + 1) c^2 &= 0 \\
          b_1^2 + b_2^2 + b_3^2 - 2 c^2 &= 0 \\
          a_3^2 + b_3^2 - c^2 &= 0
        \end{align*}
  \item Finally, one of the two fibrations associated to $b_1^2 + b_2^2 + b_3^2 = 2 c^2$ \\
        with $t = (b_1 + i b_2)/(\sqrt{2} c + b_3) = (\sqrt{2} c - b_3)/(b_1 - i b_2)$:
        \begin{align*}
          2 t b_1 - (t^2 - 1) b_3 - \sqrt{2} (t^2 + 1) c &= 0 \\
          2 t b_2 + i (t^2 + 1) b_3 + i \sqrt{2} (t^2 - 1) c &= 0 \\
          4 t^2 a_1^2 + (t^2 - 1)^2 b_3^2 + 2 \sqrt{2} (t^4 - 1) b_3 c + 2 (t^4 + 1) c^2 &= 0 \\
          4 t^2 a_2^2 - (t^2 + 1)^2 b_3^2 - 2 \sqrt{2} (t^4 - 1) b_3 c
            - 2 (t^4 + 1) c^2 &= 0 \\
          a_3^2 + b_3^2 - c^2 &= 0
        \end{align*}
\end{itemize}

The genus~$2$ quotient of the third fibration (which is defined over~$\Q$)
is birational to
\begin{align*}
  y^2 &= -2 (t^4 - 1) \bigl((t^8 - 1) x^6 + 4 (3 t^8 + 2 t^4 + 3) x^5 + 30 (t^8 - 1) x^4 \\
      & \qquad{} + 8 (5 t^8 - 2 t^4 + 5) x^3 + 30 (t^8 - 1) x^2 + 4 (3 t^8 + 2 t^4 + 3) x + (t^8 - 1)\bigr)
\end{align*}
(whose right hand side splits completely over~$\Q(i)$).


\section{K3 quotients} \label{SectK3}

Projecting away from one of the coordinate points in~$\BP^6$, which is equivalent
to taking the quotient under the involution that changes the sign of the corresponding
coordinate, we obtain various K3~surfaces given as intersections of three quadrics
in~$\BP^5$ that are doubly covered by~$\bar{S}$. We will denote them by
$\bar{K}_{a_1}, \ldots, \bar{K}_c$ and their minimal desingularizations by
$K_{a_1}, \ldots, K_c$, indexed by the coordinate we `forget'. The action of $\Aut(\bar{S})$
implies that the three $K_{b_j}$ are isomorphic over~$\Q$ and that the three $K_{a_j}$
are isomorphic over~$\Q$ and isomorphic to~$K_c$ over~$\Q(i)$.

We will consider $K_c$ in the following. Its singular model~$\bar{K}_c$ is defined by
\[ a_1^2 + a_2^2 = b_3^2, \quad a_1^2 + a_3^2 = b_2^2, \quad a_2^2 + a_3^2 = b_1^2 \,. \]
There are 12~nodes on this model (whose preimages are 24 of the 48~nodes on~$\bar{S}$;
the other 24~nodes disappear under the quotient map). The cover is branched
over the union of eight conics given by $a_1^2 + a_2^2 + a_3^2 = 0$.
This K3~surface, which parametrizes `Euler bricks', has an extra involution
(already known to Euler) given by
\[ [a_1, a_2, a_3, b_1, b_2, b_3]
     \longmapsto [a_2 a_3, a_1 a_3, a_1 a_2, a_1 b_1, a_2 b_2, a_3 b_3] \,.
\]
This involution of~$K_c$ does not descend to~$\bar{K}_c$; in particular,
it does not lift to an automorphism of~$S$.

Let $\pi \colon S \to K_c$ be the the map
induced by dividing~$S$ by the action of the sign change~$\sigma_c$ on~$c$. Note that
$\pi$ factors as $S \to S/\langle \sigma_c \rangle \to K_c$, where the first map
is a double cover ramified exactly on the eight conics whose image on~$\bar{S}$
is contained in the hyperplane $c = 0$, and the second map contracts the images
of the exceptional curves corresponding to singular points with $c = 0$ on~$\bar{S}$.
Let $\cE_\pi$ denote the set of these exceptional curves on~$S$. From the
factorization of~$\pi$ above, we easily deduce the following, which will be
used in Section~\ref{SectCurves} below.

\begin{lemma} \label{L:pi}
  For~$C$ a class in~$\Pic S$, we have
  \[ \pi^* \pi_* C = C + \sigma_c(C) + \sum_{E \in \cE_\pi} (C \cdot E) E \,. \]
\end{lemma}

Among the quadrics containing~$\bar{K}_c$, there are exactly three of rank~3
(the defining ones) and six of rank~4 (of the form
$a_j^2 - a_k^2 + b_j^2 - b_k^2 = 0$ or
\hbox{$2 (a_j^2 - b_j^2) + b_1^2 + b_2^2 + b_3^2 = 0$}).
In the same way as for the genus~5 fibrations on~$S$,
this defines $3 + 2 \cdot 6 = 15$ elliptic fibrations on~$K_c$.

We find that the images on~$K_c$ of the curves in~$\cG$ generate a subgroup~$G$
of rank~20, and therefore of finite index, in~$\Pic K_c$. The determinant
of the intersection pairing is~$-32$, so the subgroup is saturated at all primes
except possibly~2. We can check that the natural map $\pi^* \colon \Pic K_c \to \Pic S$
induces an injection \hbox{$G/2G \to \Pic S/2\Pic S$}, which shows that~$G$ is also
2-saturated, whence $G = \Pic K_c$. This proves the following.

\begin{lemma} \label{L:Kc_even}
  The Picard group of~$K_c$ is generated by the images of the curves in~$\cG$.
  In particular, every class in~$\Pic K_c$ has even intersection with the
  class of hyperplane sections.
\end{lemma}

So there are no curves of odd degree on~$\bar{K}_c$.

\begin{definition}
An \emph{integral curve}~$C$ on a scheme~$X$ is an irreducible and reduced curve $C\subset\bar X$.
\end{definition}

This notion is unrelated to integral curves in the theory of differential equations (see, e.g., \cite{Vojta}).

Regarding curves of degrees 2 and~4, we have the following.

\begin{lemma} \label{L:Kc_lowdeg}
  Let $C \subset \bar{K}_c$ be an integral curve.
  \begin{enumerate}[\upshape(1)]
    \item If~$C$ is a plane curve, then~$C$ is a conic, and all conics on~$\bar{K}_c$
          are images of curves in~$\cG$; there are 44 such conics.
    \item If~$C$ spans a~$\BP^3$, then $\deg C = 4$ and~$C$ is a fiber of one of
          the 15~elliptic fibrations mentioned above.
    \item If $\deg C = 4$ and~$C$ is not contained in a~$\BP^3$, then~$C$ is a
          smooth rational normal curve that meets the branch locus transversally
          in eight points. There are 56 such curves on~$\bar{K}_c$.
  \end{enumerate}
\end{lemma}

\pagebreak[2] 

\begin{proof} \strut
  \begin{enumerate}[(1)]
    \item If~$C$ is a plane curve, then~$C$ is (contained in) an intersection of
          quadrics in the plane it spans. Since there are no curves of odd degree
          on~$\bar{K}_c$,~$C$ must be a conic.

          By enumerating all elements~$C$ of~$\Pic K_c$ with $C \cdot H' = 2$ (where $H'$
          is the hyperplane class) and $C^2 = -2$ and excluding those with negative intersection
          with the curves we know, we find that the only conics on~$\bar{K}_c$ are those obtained
          as images of curves in~$\cG$. There are three different kinds (each forming one orbit
          under the subgroup of~$\Aut(K_c)$ induced by~$\Aut(S)$), namely (i) the eight
          conics whose union is the branch locus, (ii) twelve conics in the hyperplanes $a_j = 0$,
          and (iii) 24~conics obtained as quotients of the genus~1 curves in~$\cG$ that are
          invariant under the involution that changes the sign of~$c$.
    \item As before,~$C$ is contained in an intersection of quadrics in the~$\BP^3$ it spans.
          Since $\deg C$ is even and conics span planes, we must have $\deg C = 4$
          and $p_a(C) = 1$, so $C^2 = 0$. By enumerating all relevant elements of~$\Pic K_c$
          as before, we find that they all correspond to fibers of one of the 15~fibrations.
    \item Any curve of degree~4 that spans~$\BP^4$ is a rational normal curve.
          We then have $C^2 = -2$. We can enumerate the relevant elements again
          and check that their intersection with the branch locus is~8. There are 80
          such classes, falling into orbits under~$\Aut(K)$ of sizes 8, 24 and~48.
          The orbit of size~8 comes from the images of the conics in the branch locus
          under the extra involution. The orbit of size~24 does not come from curves,
          since its elements have intersection multiplicity~2 with one of the exceptional
          curves (so the curve would have to have a singularity there). The orbit of
          size~48 comes from curves in hyperplane sections given by
          $a_1 + \sqrt{2} i a_2 + a_3 + b_2 = 0$ and their images under sign changes
          and permutations.
    \qedhere
  \end{enumerate}
\end{proof}

If $C \subset \bar{K}_c$ is an integral curve of degree~6 contained in a~$\BP^4$,
then the residual intersection with any hyperplane that contains~$C$ is a conic.
If the conic is not contained in the branch locus, then it meets four of the eight
branch conics transversally in one point each. So the residual sextic also has to meet
these four branch conics transversally in one point each (note that the branch
locus is disjoint from the singularities of~$\bar{K}_c$). (It will also meet the
other four branch conics with multiplicity~$2$.) The pull-back of the sextic to~$\bar{S}$ will
be irreducible of degree~$12$ unless the sextic meets the branch locus with
even multiplicity in each intersection point. So, for the sextic to lift to a sextic
on~$\bar{S}$, it has to pass through two points in which two of the branch conics
intersect. This means that the hyperplane has to pass through these points as well.
Note that the four branch conics split into two pairs such that the conics in each
pair intersect in two distinct points, whereas the conics in distinct pairs
do not intersect.

There are two orbits of non-branch conics under the subgroup of~$\Aut(\bar{K}_c)$
induced by~$\Aut(\bar{S})$. For one of these orbits, the intersection points
of the non-branch conic with the four relevant branch conics are distinct from
the intersection points of the branch conics. So, in this case, we need to
consider hyperplanes containing the non-branch conics and one intersection point
for each pair of branch conics. Of these four hyperplanes, only two contain a
sextic curve (the other two contain another conic), and each of these sextics
meets the branch locus in some points with multiplicity one, and so lifts to
an irreducible curve of degree~$12$ on~$\bar{S}$.

For the other orbit, the non-branch conic intersects the branch conics at
intersection points of the pairs (one intersection point for each pair).
In this case, the hyperplane must either pass through the other intersection
point of a pair or contain the tangent lines at the intersection point
of the two branch conics. Three of the possible four combinations of
conditions lead to a unique hyperplane, whose intersection with~$\bar{K}$
we can check does not contain an integral sextic curve. The last possibility
(which occurs when the hyperplane contains the tangent lines at both intersection
points) leads to a pencil of hyperplanes. However, the intersection of~$\bar{K}$
with their common~$\BP^3$ already contains the initial non-branch conic
with multiplicity~$2$, so none of the hyperplanes can contain a sextic on~$\bar{K}$.

If the conic is a branch conic, then we can check explicitly that no hyperplane
containing the~$\BP^2$ spanned by the conic is tangent to all four branch conics
that do not meet the original one. This again implies that the pull-back to~$\bar{S}$
of the residual sextic ramifies.

\medskip

We obtain the following result.

\begin{lemma} \label{L:Kc_deg6P4}
  If $C \subset \bar{K}_c$ is an integral sextic curve that is contained in a~$\BP^4$,
  then~$C$ lifts to an irreducible curve of degree~12 on~$\bar{S}$.
\end{lemma}


\section{Curves of low degree on~$\bar{S}$} \label{SectCurves}

We can use our explicit knowledge of the Picard group to determine the set
of curves on~$S$ or~$\bar{S}$ of small degree.
We freely identify curves on~$\bar{S}$ with their strict transforms on~$S$
and with their classes in~$\Pic S$.

\begin{theorem} \label{T:smalldeg} \strut
  \begin{enumerate}[\upshape(1)]
    \item \label{conics}
          All conics on~$\bar{S}$ are contained in~$\cG$.
    \item \label{rat4}
          The surface~$\bar{S}$ does not contain smooth rational curves of degree~4.
    \item \label{elliptic}
          All curves of degree~4 and arithmetic genus~1 on~$\bar{S}$ are in~$\cG$.
          In particular, all such curves are smooth and hence of geometric
          genus~1.
  \end{enumerate}
\end{theorem}

\begin{proof} \strut
  \begin{enumerate}[(1)]
    \item Any conic~$C$ in~$\bar{S}$ must be smooth, since~$\bar{S}$ does not contain
          curves of odd degree. We have $C \cdot K_S = 2$ and $C^2 = -4$ by the
          adjunction formula. We can enumerate all lattice points in the Picard lattice
          satisfying these two conditions; this results in 2048~elements. If~$C$ is
          a curve, then it has to have nonnegative intersection with all the curves
          in~$\cG$ (except possibly itself). Testing this condition leaves only
          the 32 known conics in~$\cG$.
    \item If~$C$ is a smooth rational curve of degree~4 on~$\bar{S}$, then~$C$ spans
          a~$\BP^4$ (otherwise $p_a(C) = 1$, see Theorem~\ref{T:no_plane_curves} below),
          which we will denote~$P_4$.
          The image of~$C$ in~$\bar{K}_c$ then either is isomorphic to~$C$ and hence
          a rational normal curve of degree~4. But these all lift to hyperelliptic
          curves of genus~3 on~$\bar{S}$ by Lemma~\ref{L:Kc_lowdeg}. Or else the image
          spans a~$\BP^3$; this will be the case when $P_4$
          contains the point $[0,0,0,0,0,0,1]$. We can consider $K_{a_j}$ instead
          of~$K_c$ for any $j \in \{1,2,3\}$, so the only remaining cases are those
          when $P_4$ contains the $\BP^3$ given by $b_1 = b_2 = b_3 = 0$.
          Projecting away from this~$\BP^3$, we obtain the morphism
          $S \to Q \cong \BP^1 \times \BP^1$ mentioned in Section~\ref{SectMod}.
          Then $P_4$ must be the pull-back of one of the lines on~$Q$, and
          $P_4 \cap \bar{S}$ is a fiber of one of the two corresponding fibrations.
          But no such fiber contains a rational normal curve of degree~4.
    \item This statement is proved in the same way as the result on conics.
          A curve~$C$ of degree~4 and arithmetic genus~1 satisfies $C^2 = -4$.
          We can enumerate the relevant 16\,160 classes in~$\Pic S$ and check
          that none of them has nonnegative intersection with all curves in~$\cG$.
    \qedhere
  \end{enumerate}
\end{proof}

We now show that we know all curves on the surface~$\bar{S}$ that are
contained in a low-dimensional linear subspace.

\begin{theorem} \label{T:no_plane_curves}
  Let $C \subset \bar{S}$ be an integral curve.
  \begin{enumerate}[\upshape(1)]
    \item If~$C$ is contained in a plane, then~$C$ is a conic; in particular, $C \in \cG$.
    \item If~$C$ spans a~$\BP^3$, then~$C$ is one of the genus~1 curves in~$\cG$.
    \item If~$C$ spans a~$\BP^4$, then~$C$ has degree~8 and is a fiber of one of
          the 28~fibrations described in Proposition~\ref{prop:genus5fibrations}. In particular,~$C$ is
          either a smooth canonical curve of genus~5, or else it is one of the
          96~hyperelliptic curves of genus~3 occurring as irreducible singular fibers.
  \end{enumerate}
\end{theorem}

\begin{proof} \strut
  \begin{enumerate}[(1)]
    \item Let $P_2$ be the plane spanned by~$C$ ($C$ cannot be a line, since there
          are no curves of odd degree on~$\bar{S}$). Then~$C$ is an intersection
          of quadrics in~$P_2$, so~$C$ is a conic. By Theorem~\ref{T:smalldeg},
          it follows that $C \in \cG$.
    \item Let $P_3$ be the linear subspace spanned by~$C$. Then~$C$ is contained
          in an intersection of quadrics in~$P_3$; in particular, the degree of~$C$
          is at most~4 (and even). Since~$C$ is not a plane curve, the degree
          must be~4, hence by Theorem~\ref{T:smalldeg}, $C \in \cG$.
    \item Let $P_4$ be the linear subspace spanned by~$C$. Again~$C$ is contained
          in an intersection of quadrics, so the degree of~$C$ is now 6 or~8.
          If the degree were~6, then the image of~$C$ on~$\bar{K}_c$ would be
          either a cubic (which is impossible) or a sextic contained in a hyperplane.
          But Lemma~\ref{L:Kc_deg6P4} shows that curves of the latter type do not
          come from sextics on~$\bar{S}$. So the degree of~$C$ is~8, and $C$
          is an intersection of three quadrics. Therefore the four quadrics
          defining~$\bar{S}$ will be linearly dependent when restricted to~$P_4$,
          which means that there is a quadric~$Q$ containing~$\bar{S}$ that contains~$P_4$.
          But a quadric in~$\BP^6$ that contains a~$\BP^4$ has rank at most~4.
          So~$Q$ is either one of the six rank~3 quadrics or one of the eleven
          rank~4 quadrics vanishing on~$\bar{S}$. In the former case, $P_4$ is
          the preimage of a point on the conic obtained by projecting~$Q$, in the
          latter case, $P_4$ is the preimage of a line on the quadric in~$\BP^3$
          obtained by projecting~$Q$. In either case, we see that $C = \bar{S} \cap P_4$
          is a fiber of one of the associated fibrations. The bad fibers have been
          described in Section~\ref{SectFib} in the text following
          Proposition~\ref{prop:genus5fibrations}.
     \qedhere
  \end{enumerate}
\end{proof}

\begin{theorem} \label{T:no_sextics}
  There are no integral curves of degree~6 on~$\bar{S}$.
\end{theorem}

\begin{proof}
  Let $C \subset \bar{S}$ be an integral curve of degree~6. By Theorem~\ref{T:no_plane_curves},
  we know that~$C$ spans a~$\BP^5$ or~$\BP^6$; it follows that $p_a(C) \in \{0, 1\}$, so
  $C^2 \in \{-8, -6\}$. The image~$C'$ of~$C$ on~$\bar{K}_c$ will still be of degree~6
  (there are no curves of odd degree on~$K_c$) and span~$\BP^5$, since by
  Lemma~\ref{L:Kc_deg6P4}, no curve of degree~6 in a hyperplane section of~$\bar{K}_c$
  is the image of a sextic on~$\bar{S}$. So $p_a(C') \in \{0, 1\}$ as well, and
  $(C')^2 \in \{-2, 0\}$. We enumerate all classes~$C'$ in~$\Pic K_c$ having
  intersection~6 with the hyperplane section and self-intersection $-2$ or~$0$,
  such that $C'$ has nonnegative intersection with all curves obtained as images
  from curves in~$\cG$. This results in 1088 such classes with $(C')^2 = 0$
  and 1680~classes with $(C')^2 = -2$, which fall into 14 and 20~orbits, respectively,
  under the group of linear automorphisms of~$\bar{K}_c$. If such a class~$C'$
  is the image of an integral curve~$C$ on~$\bar{S}$, then by Lemma~\ref{L:pi},
  we must have
  \[ \pi^* C' = \pi^* \pi_* C
              = C + \sigma_c(C) + \sum_{E \in \cE_\pi} (C \cdot E) E
              = \tau(C) \,,
  \]
  where $\tau \colon \Pic S \to \Pic S$ is the endomorphism defined by the expression above.
  For a representative~$C'$ of each of the relevant orbits, we first check if $\pi^* C'$
  is in the image of~$\tau$ (otherwise $C'$ is not in the image of~$\pi_*$).
  If this is the case, then we find the (finitely many) `candidate classes'~$C \in \Pic S$
  satisfying (i) $\tau(C) = \pi^* C'$, (ii) $C^2 \in \{-6, -8\}$, (iii) $C \cdot D \ge 0$
  for all $D \in \cG$. In all cases, it turns out that the set of candidate classes
  is empty. This shows that there are no integral curves of degree~6 on~$\bar{S}$
  that span a~$\BP^5$ or~$\BP^6$, which proves the claim.
\end{proof}

\begin{cor}
  The set $\cG$ consists precisely of the integral curves~$C$ on~$S$ such that
  $C \cdot K_S \le 6$.
\end{cor}

\begin{proof}
  We know that all curves in~$\cG$ satisfy the conditions.
  Observe that for an integral curve~$C$ on~$S$,
  we must have $C \cdot K_S \in \{0, 2, 4, \dots\}$. So if $C \cdot K_S \le 6$,
  we have $C \cdot K_S \in \{0, 2, 4, 6\}$. If the intersection number vanishes,
  then~$C$ is an exceptional curve and so $C \in \cG$. If $C \cdot K_S = 2$,
  then the image of~$C$ on~$\bar{S}$ is a conic, and so $C \in \cG$ by
  statement~\eqref{conics} in Theorem~\ref{T:smalldeg}.
  If $C \cdot K_S = 4$, then the image of~$C$ on~$\bar{S}$ is a curve of degree~4,
  which has arithmetic genus either
  zero (then~$C$ is a smooth rational curve), one or three. The first two cases are
  taken care of by statements \eqref{rat4} and~\eqref{elliptic} in Theorem~\ref{T:smalldeg},
  respectively. In the last case,~$C$ is a plane quartic, and this is ruled out
  by Theorem~\ref{T:no_plane_curves}. Finally by Theorem~\ref{T:no_sextics},
  $C \cdot K_S \neq 6$.
\end{proof}

Even though there are no plane quartics on~$\bar{S}$, we can find some more
curves of genus~3. For this, we consider the quotient $S \to K_c$ we studied
in Section~\ref{SectK3}. The extra involution on~$K_c$ does not lift to an
automorphism of~$S$, which allows us to produce
new curves by projecting a known curve to~$K_c$, applying the involution and pulling
the result back to~$S$. For example, taking one of the conics in the branch locus
and applying the involution, we obtain a smooth rational curve of degree~4
meeting the branch locus in eight points. Its pull-back to~$\bar{S}$ therefore is
a hyperelliptic curve of genus~3. It is of degree~8 and contained in a hyperplane
section cut out by $b_1 \pm b_2 \pm b_3 = 0$; the residual intersection is another
one of these hyperelliptic curves. Since the hyperplane $c = 0$ meets the curve
exactly in the eight Weierstrass points, the curve is bicanonically embedded.
It is smooth and has self-intersection~$-4$. The eight curves obtained in this
way together with the further 24 similar curves obtained from replacing~$c$
by one of the~$a_j$ form an orbit under the automorphism group. They are all isomorphic
to the `octahedral' curve
\[ y^2 = x^8 + 14 x^4 + 1 \,. \]

Pulling back the other rational quartics from~$\bar{K}_c$ to~$\bar{S}$ (and similarly
from the~$\bar{K}_{a_j}$), we find 192 further smooth hyperelliptic genus~3 curves
of degree~8 and spanning a~$\BP^5$. They sit in hyperplane sections of the form
\[ a_1 + \sqrt{2} i a_2 + a_3 + b_2 = 0 \]
(modulo the action of~$\Aut(S)$) and are isomorphic (over~$\Q(i,\sqrt{2})$) to
\begin{align*}
  y^2 &= x^8 + 4 x^6 + 32 x^5 + 86 x^4 + 64 x^3 + 36 x^2 + 32 x + 17 \,.
\end{align*}

We find some further interesting curves by considering images of graphs of automorphisms
of~$X$ under the map $X \times X \to \bar{S}$. The centralizer of~$G_0$ in~$\Aut(X)$
has order~32. There are three times 32 further elements that centralize one of the
subgroups of~$G_0$ given by sign changes of the type $(-,+,+)$ or~$(+,-,-)$. The
remaining 64~elements of~$\Aut(X)$ only commute with the sign change~$(-,-,-)$.
This implies that among the images in~$\bar{S}$ of these 192~curves, we find the
32~conics, three times 16 genus~1 curves (these are the genus~1 curves in~$\cG$
that are not contained in one of the hyperplanes $b_j = 0$) and finally 16~curves
isomorphic to~$X$ divided by the central element, which results in the Fermat quartic.
Its images on~$\bar{S}$ are bicanonically embedded; in particular, they span a
hyperplane, which is of the form $a_1 \pm a_2 \pm a_3 \pm ic = 0$ and contains two
of these 16~curves. Since they have negative self-intersection,
both the hyperelliptic and the non-hyperelliptic genus~3 curves
are not contained in the cone in~$\Pic S$ spanned by the curves in~$\cG$.

\medskip

The partial results above suggest the following question.

\begin{question}
  Are all the curves of (geometric) genus at most~1 on~$S$ contained in~$\cG$?
\end{question}

We can say the following. We write~$E$ for the exceptional divisor on~$S$.

\begin{lemma} \label{L62} \strut
  \begin{enumerate}[\upshape(1)]
    \item Any rational curve~$C$ on~$\bar{S}$ must satisfy $C \cdot E \ge 6$.
    \item Any curve~$C$ of geometric genus~1 on~$\bar{S}$ must satisfy $C \cdot E \ge 2$.
  \end{enumerate}
\end{lemma}

\begin{proof}
  Recall that~$Y$ as defined near the end of Section~\ref{SectMod} is a smooth
  double cover of~$\bar{S}$ branched exactly above the singular points of~$\bar{S}$.
  Note that~$Y$ maps to a genus~2 curve~$C_2$ with fibers that are isomorphic to the
  curve~$X$ of genus~5. So~$Y$ does not contain curves of genus $< 2$.
  Let~$C$ be a curve of geometric genus~$g$ on~$\bar{S}$ and let $\tilde{C} \to C$
  be its desingularization. The pull-back $\tilde{C} \times_C Y$ must then have
  geometric genus at least~$2$,
  it must therefore be ramified in at least six points when $g = 0$ and in at least two
  points when $g = 1$. This translates into the statements on intersection numbers with~$E$.
\end{proof}

In fact, the conics pass through six singularities each, so they lift to 32~sections of
the fibration of~$Y\!$, whereas the genus~1 curves in~$\cG$
lift to curves of genus 5 or~9 on~$Y\!$. Since~$Y$ maps to $C_2 \times C_2$,
any curve of genus~2 in~$Y$ must be isomorphic to~$C_2$ (and map to the
graph of an automorphism of~$C_2$). If $C' \subset Y$
is a curve isomorphic to~$C_2$, then its preimage in~$X \times X$ is a four-fold
\'etale cover of~$C'$ all of whose components must have genus at least~5.
This implies that the preimage is connected and isomorphic to~$X$, so
it is the graph of an automorphism~$\sigma$ of~$X$ (since it cannot be a fiber
of one of the projections). For $\sigma$ to arise in this way, the group~$G_0^+$
must stabilize its graph, which means that $\sigma$ is in the centralizer
of~$G_0$ in~$\Aut(X)$. This centralizer has order~32, so all of these curves
are accounted for by the lifts of the 32~conics on~$\bar{S}$. This implies
that all other curves must lift to curves of genus at least~$3$ on~$Y$,
and we obtain the following improvement of Lemma~\ref{L62}.

\begin{lemma} \label{L:bounds} \strut
  \begin{enumerate}[\upshape(1)]
    \item A rational curve~$C$ on~$\bar{S}$ that is not a conic must satisfy $C \cdot E \ge 8$.
    \item A curve~$C$ of geometric genus~1 on~$\bar{S}$ must satisfy $C \cdot E \ge 4$.
  \end{enumerate}
\end{lemma}

The first statement recovers a result from~\cite{GFU}.

\medskip

Note that surfaces of general type are conjectured to have only finitely
many curves of genus at most~1. In general, this is only known in very
few cases, for instance, when Bogomolov's inequality $(K_S)^2 > c_2(S)$
holds (see~\cite{Bo}); for surfaces contained in an abelian variety
(see~\cite{Fa}); for very general surfaces of large degree in projective
space (see work of Demailly, Siu, Diverio-Merker-Rousseau, B\'erczi on
the Green-Griffiths Conjecture).
None of these cases covers the
surface~$S$: Bogomolov's inequality fails for~$S$ since $(K_S)^2 = 16 < 80 = c_2(S)$;
the surface~$S$ is simply connected and hence is not contained in an abelian
variety; the surface~$S$ is not a surface in $\mathbb{P}^3$, since the only such
surfaces of general type with primitive canonical divisor are the surfaces of degree
five, and a plane section of a quintic is a curve of odd degree.

According to the Bombieri-Lang conjecture, the rational points on a variety
of general type are not Zariski-dense. This means that all but finitely
many rational points on a surface of general type lie on a finite set
of curves of genus zero or one on the surface. If the question above has
a positive answer, then the only such curves
defined over~$\Q$ on~$\bar{S}$ are the
conics corresponding to degenerate cuboids. So
the Bombieri-Lang conjecture would then imply that there are only finitely
many distinct rational cuboids (up to scaling).


\end{document}